\documentclass[final]{siamltex1213}

\usepackage{amssymb}
\usepackage{amsmath}
\usepackage{bbm}

\renewcommand{\parallel}{\|}
\newtheorem{mydef}[theorem]{Definition}
\DeclareMathOperator*{\argmin}{\arg\!\min}
\def\F{{\mathcal{F}}}
\def\N{{\mathcal{N}}}

\def\T{{\mathcal{T}}}

\def\W{{\mathcal{W}}}
\def\E{{\mathcal{E}}}
\def\Ea{{\mathcal{E}^{\alpha}}}
\def\l{{\mathbf{l}}}
\def\u{{\mathbf{u}}}
\def\v{{\mathbf{v}}}
\def\w{{\mathbf{w}}}
\def\x{{\mathbf{x}}}
\def\y{{\mathbf{y}}}
\def\z{{\mathbf{z}}}
\def\q{{\mathbf{q}}}
\def\p{{\mathbf{p}}}
\def\RA{{\mathbb{R}_{A}^d}}
\def\CJ{{\mathbb{C}_{J}^d}}

\title{Rank deficiency of Kalman error covariance matrices in linear time-varying system with deterministic evolution}

\author{
Karthik S. Gurumoorthy\footnotemark[3]\ \footnotemark[2] 
\and  Colin Grudzien\footnotemark[4]  \footnotemark[5] \footnotemark[6]
\and  Amit Apte\footnotemark[3]\ \footnotemark[2] 
\and  Alberto Carrassi\footnotemark[5] \footnotemark[7]
\and  Christopher K. R. T. Jones\footnotemark[4] 
}

\begin{document}
\maketitle \slugger{sicon}{xxxx}{xx}{x}{x--x}

\renewcommand{\thefootnote}{\fnsymbol{footnote}}

\footnotetext[3]{International Center for Theoretical Sciences, Tata
  Institute of Fundamental Research, Bangalore, Karnataka, India}

\footnotetext[4]{Department of Mathematics, University of North
  Carolina, Chapel Hill, North Carolina, USA}

\footnotetext[5]{Nansen Environmental and Remote Sensing Center,
  Bergen, Norway.}

\footnotetext[2]{This work benefited from the support of the
  AIRBUS Group Corporate Foundation Chair in Mathematics of Complex
  Systems established in ICTS-TIFR.}
  
 \footnotetext[6]{This work was partially funded by the project REDDA of the Norwegian Research Council under contract 250711.}
 
  \footnotetext[7]{This work was supported by the EU-FP7 projects
  SANGOMA under grant contract 283580 and also with the help of funding from the 
  centre of excellence EmblA of the Nordic Countries research council - NordForsk.}

\begin{abstract} We prove that for linear, discrete, time-varying, deterministic system (perfect model) with noisy outputs, 
the Riccati transformation in the Kalman filter 
  asymptotically bounds the rank of the forecast and the analysis
  error covariance matrices to be less than or equal to the number of
  non-negative Lyapunov exponents of the system. Further, the support of
  these error covariance matrices is shown to be confined to the space
  spanned by the unstable-neutral backward Lyapunov vectors, providing the
  theoretical justification for the methodology of the algorithms that
  perform assimilation only in the unstable-neutral subspace. The equivalent
  property of the autonomous system is investigated as a special case.
\end{abstract}

\pagestyle{myheadings} \thispagestyle{plain}
\markboth{KARTHIK GURUMOORTHY et al.}{RANK DEFICIENCY OF ERROR COVARIANCE MATRICES}

\section{Introduction}

The problem of estimating the state of an evolving system from an
incomplete set of noisy observations is the central theme of the
state estimation and optimal control theory
\cite{Jazwinski}, also referred to as data assimilation (DA) in geosciences \cite{GhilMalanotte1991,Talagrand1997}.
In the filtering procedure, based on the concept of recursive processing,
measurements are utilized sequentially, as they become available
\cite{Jazwinski}. For linear dynamics, and when a linear relation
exists between measurements and the state variables, and when the
errors associated to all sources of information are Gaussian, the solution can be expressed via the Kalman filter (KF) equations
\cite{Kalman1960}. The KF provides a closed set of equations for the
first two moments of the posterior probability density function of the
system state, conditioned on the observations. 
In the case of nonlinear dynamics, the first
order extension of the KF is known as the extended Kalman filter
(EKF) \cite{Jazwinski}, whereas a Monte Carlo approximation
is the basis of a set of methods known as Ensemble Kalman filter
(EnKF) both of which have been studied extensively in geophysical contexts \cite{Miller1994, Evensen}. 
 
Atmosphere and ocean are example of dissipative chaotic systems. This
implies the sensitivity to initial condition \cite{Lorenz1963} and the
fact that the estimation error strongly projects on the unstable
manifold of the dynamics \cite{Pires-et-al-1996} which has inspired the development of a class of
algorithms known as assimilation in the unstable subspace (AUS)
\cite{TrevisanUboldi2004}. In AUS, the span of the leading Lyapunov
vectors (to be defined precisely in later sections), or a suitable
approximation of this span, is used explicitly in the analysis step: the
analysis update is confined to the unstable subspace
\cite{Palatella2013}. 
AUS has been formalized in the framework of the EKF,
EKF-AUS \cite{TrevisanPalatella2011}, and in the variational
(smoothing) procedure, 4DVar-AUS \cite{Trevisan2010}. 
Applications with atmospheric, oceanic, and
traffic models \cite{UboldiTrevisan2006, Carrassi-etal2008,
  PalatellaTrevisanRambaldi2013} showed that even in high-dimensional
systems, an efficient error control is achieved by monitoring only a
limited number of unstable directions, making AUS a computationally 
efficient alternative to standard procedures.
The AUS methodology is based on and at the same
time hints at a fundamental observation: the span of the estimation
error covariance matrices asymptotically (in time) tends to the
subspace spanned by the unstable-neutral Lyapunov vectors.

The search for a formal proof of this aforesaid property is the basic motivation of the present work which is
focused on linear non-autonomous, and linear autonomous perfect-model
dynamical systems. The main results of the paper are as
follows. In Theorem~\ref{thm:eigenvalueszero} we show that the error covariance matrices, independent of the initial condition, asymptotically become rank deficient in time and then in Theorem~\ref{thm:epsilonnull} we characterize their null spaces by proving that the restriction of the these matrices onto the stable backward Lyapunov vectors converges to zero in time. When restricted to the linear, autonomous system with the time invariant propagator $A$, we establish that the stable space of the time independent backward Lyapunov vectors equal the stable space of $A^T$---span of generalized eigen-vectors of $A^T$ corresponding to eigen-values less than one in absolute magnitude---in Theorem~\ref{thm:eigspaceequality}. Consequently, in Corollary~\ref{thm:nullspaceS1AT} we show that the null space of the error covariance matrices contain the stable space of $A^T$ asymptotically.

The paper is organized as follows. After describing the general
notation in Section~\ref{sec:notations}, the non-autonomous case is
considered in Section~\ref{sec:nonautosys}. The assumptions used in
proving our main result, other useful results such as the Oseledet
theorem, and the concepts of observability and controllability for
noiseless systems are described in Sections~\ref{sec:nonaut-setup},
\ref{sec:oseledet} and \ref{sec:control}. The Theorem~\ref{thm:eigenvalueszero} discussing the rank deficiency of error covariance matrices is presented in Section~\ref{sec:errcovrankdeficiency} and the proof of the Theorem~\ref{thm:epsilonnull} using the geometric viewpoint of
Kalman filtering \cite{Bougerol93, wojtowski2007, Bonnabel13}, is
detailed in Section~\ref{sec:errcovnullspace}. Section~\ref{sec:errcovnumer}
presents some numerical results buttressing the
theorem. Section~\ref{sec:autosys} includes the proof of
Corollary~\ref{thm:nullspaceS1AT} along with a numerical illustration
supporting the analytical findings for autonomous systems. We conclude in Section~\ref{sec:discussion}.

Although the extension of these results to the general
nonlinear case is the object of active research \cite{sanz-alonso2014},
the current findings already provide a formal justification to the AUS
foundation and further motivates its use as a DA 
strategy in nonlinear chaotic dynamics.

\section{Notations} \label{sec:notations}
The dimension of the state space is represented by $d$. For any square
matrix $Z \in \mathbb{C}^{d \times d}$ let the set $\{
\lambda_1(Z),\cdots,\lambda_d(Z)\}$ represent the eigen-values of $Z$
where $|\lambda_1(Z)| \geq \cdots \geq |\lambda_d(Z)|$. Similarly, let
the set $\{\sigma_1(Z),\cdots,\sigma_d(Z)\}$ stand for the
singular-values of $Z$ with $\sigma_1(Z) \geq \cdots \geq
\sigma_d(Z)$. We define the column vectors of the matrix $V_Z =
\left[\v_1(Z),\cdots,\v_d(Z)\right] $ to be the generalized
eigen-vectors of $Z$ of satisfying the relation $Z V_Z = V_Z J(Z)$ where
$J(Z)$ is the Jordan-canonical form of $Z$. In the event that $Z$ is
diagonalizable ($J(Z)$ is diagonal), let the entries of the diagonal
matrix $\Lambda_Z = J(Z)$ symbolize the eigen-values of
$Z$ and the columns of $V_Z$---the eigen-vectors---be of unit magnitude. $Z^{\ast}$ denotes the adjoint of $Z$ for the scalar product
under consideration in $\mathbb{C}^d$ and $Z^{\dagger}$ represents the
conjugate transpose of $Z$. For the canonical scalar product $\langle
\u,\v\rangle = \u^{\dagger} \v$ in $\mathbb{C}^d$, $Z^{\ast} =
Z^{\dagger}$ and when confined to the real space $\mathbb{R}^d$ where
$\langle \u,\v\rangle = \u^{T} \v$, $Z^{\ast} = Z^{T}$. Unless
explicitly stated we assume a real vector space endowed with a canonical scalar product. The matrix norm $\parallel Z \parallel$ we consider is the largest singular-value $\sigma_1(Z)$ of $Z$. 
The notation $Z>0$ ($Z \geq 0$) is used when $Z$ is symmetric, positive-definite (positive-semidefinite). For any two symmetric matrices $Z_1$, $Z_2$, the notation $Z_1 \geq Z_2$ means $Z_1 - Z_2 \geq 0$. 
The following definitions are useful.
\begin{mydef}[Real span]
\label{def:realspan} The real span of a complex vector $\w = \u + i
\v$ where $\u,\v \in \mathbb{R}^d$ is the vector space $\T_{\w}
\subset \mathbb{R}^d$ defined as
\begin{equation*}
  \T_{\w} \equiv \{ \alpha \u + \beta \v: \alpha,\beta
  \in \mathbb{R}\}.
\end{equation*}
\end{mydef}
\begin{mydef}[$\alpha$-eigenspace]
\label{def:alphaeigenspace} Given an $\alpha > 0$, the
$\alpha$-eigenspace of a square matrix $Z$ denoted by $\Ea(Z)$ is the
real span of the generalized eigenvectors of $Z$ corresponding to
eigen-values $\lambda$ with $|\lambda| < \alpha$.
\end{mydef}

\section{Non-autonomous systems} \label{sec:nonautosys}
\subsection{Set up and Assumptions} \label{sec:nonaut-setup} 
We define the general linear non-autonomous dynamical system at time
$n\geq 0$ by
\begin{align}
  \label{eq:nonautgen} \x_{n+1} &= A_{n+1} \x_n + F_{n+1} \p_{n+1} \\
  \y_{n+1} &= H_{n+1} \x_{n+1} + \q_{n+1} \nonumber
\end{align}
where $\x_n \in \mathbb{R}^d$, $\q_n\in \mathbb{R}^q$, $\p_n
\in\mathbb{R}^p$.  The $\x_n$ are the state variables, $\p_n$
represents model noise, $\y_n$ represents observational variables and
$\q_n$ is the observational noise term.  The basic random variables
$\{\x_0, \q_1, \q_2, \cdots ,\p_1,\p_2,\cdots\}$ are all assumed to be
independent and Gaussian with
\begin{equation*} \begin{matrix} \x_0 \sim \N\left(\x_{0 \mid 0},
\Delta_0\right) & \q_n \sim \N(0,Q_n) & \p_n \sim \N(0,I) \,,
\end{matrix}\end{equation*}
such that $\Delta_0 \in \mathbb{R}^{d\times d}$ is the initial error
covariance matrix of the state variable $\x_0$, $Q_n\in
\mathbb{R}^{q\times q}$ is the observation error covariance matrix at
time $n$, and $F_n \in \mathbb{R}^{d\times p}$. The matrices $\Delta_0, Q_n, F_n, A_n, H_n$ are known for all time $n$. Further, $A_n$ and $Q_n$ are considered to be non-singular, $\parallel A_n \parallel \leq c_A$,  $\parallel Q_n \parallel \leq c_Q$, and $\parallel H_n \parallel \leq c_H, \forall n\geq1$ where $c_A$, $c_Q$ and $c_H$ are positive constants. The model noise error covariance is given by $P_n \equiv F_n F_n^T$. Unless explicitly stated $\Delta_0 > 0$, i.e, its eigen-values are strictly positive.

Filtering theory deals with the properties of the conditional
distribution, called the \emph{analysis} in the context of 
DA, of the state $\x_n$ at time $n$ conditioned on
observations $Y_{0:n} = [\y_1, \y_2, ... , \y_n]$ up to time $n$ where the first observation $\y_1$
is assumed to occur at time $n=1$. This conditional distribution provides an optimal state estimate in the
least squares sense~\cite{Jazwinski}.  Under the assumptions of
linearity and Gaussianity stated above, this conditional distribution
is Gaussian, with mean and covariance denoted by $\x_{n \mid n}$ and
$\Delta_n$ respectively:
\begin{align*} \x_{n\mid n} = \mathbb{E}[\x_n \mid Y_{0:n}] \,,
\quad\textrm{and}\quad \Delta_n = \mathbb{E}[(\x_n - \x_{n \mid
n})(\x_n - \x_{n \mid n})^T \mid Y_{0:n} ] \,.
\end{align*}
We also note that the conditional distribution, called the
\emph{forecast} in DA literature, of the state
$\x_{n+1}$ conditioned on observations $Y_{0:n}$ up to
time $n$ is Gaussian with its mean and covariance denoted by $\x_{n+1
  \mid n}$ and $\Sigma_{n+1}$ respectively:
\begin{align*} \x_{n+1 \mid n} &= \mathbb{E}[\x_{n+1} \mid Y_{0:n}]
\,,\quad\textrm{and}\quad \Sigma_{n+1} = \mathbb{E}[(\x_{n+1} -
\x_{n +1 \mid n})(\x_{n+1} - \x_{n+1 \mid n})^T \mid Y_{0:n}] \,.
\end{align*}

In this work we concern ourselves with systems that have no model error, i.e, $F_n \equiv 0$
$\forall n \geq 1$ and investigate the dynamics
\begin{align}
\label{eq:nonaut} 
\x_{n+1} = A_{n+1} \x_n \,, \quad\textrm{and}\quad
\y_{n+1} = H_{n+1} \x_{n+1} + \q_{n+1} \,.
\end{align}
We will be interested in asymptotic properties
of the conditional error covariances $\Sigma_n$ and $\Delta_n$. The Kalman filter provides a closed form,
iterative formula for obtaining these quantities~\cite{Jazwinski}. Under
the assumption of no model noise, the update equation for the forecast
error covariance is
\begin{equation}
\label{eq:Sigmaupdate} \Sigma_{n} = A_{n} \Delta_{n-1} A_{n}^T \,.
\end{equation}
By defining the Kalman gain matrix $K_n$ as
\begin{equation}
\label{eq:KalmanGainupdate} K_n \equiv \Sigma_n H_n^T \left[H_n \Sigma_n H_n^T +Q_n\right]^{-1} \,,
\end{equation}
the analysis error covariance equals
\begin{equation}
\label{eq:Deltaupdate} \Delta_n = (I - K_n H_n)\Sigma_n \,.
\end{equation}
The update equations for the means are given by
\begin{align}
\label{eq:forecastMeanUpdate}
\x_{n+1 \mid n} &= A_{n+1} \x_{n \mid n} \\
\label{eq:analysisMeanUpdate}
\x_{n+1 \mid n+1} &= \x_{n+1 \mid n} + K_{n+1}\left(y_{n+1} - H_{n+1} \x_{n+1 \mid n}\right).
\end{align}

\noindent Defining the sequence of matrices $M_n$ as
\begin{equation} 
\label{eq:Mn} 
M_1 \equiv (I - K_1 H_1) A_1, \quad M_n \equiv (I - K_n H_n) A_n M_{n-1}
\end{equation} 
and writing the propagator $B_{m:m+n}$ from time $m$ to time $m+n$ by
\begin{align}
\label{eq:Bn} B_{m:m+n} &\equiv A_{m+n} A_{m+n-1} \cdots A_{m+1},
\end{align}
the analysis covariance at time $n$ can be expressed as
\begin{align}
\label{eq:MnUpdate} \Delta_n =& (I-K_n H_n)A_n \cdots (I - K_1 H_1)
A_1 \Delta_0 A_1^T \cdots A_n^T = M_n \Delta_0 B_{0:n}^T.
\end{align}
This equation clearly shows that the asymptotic properties of
$\Delta_n$ are closely related to those of $B_{0:n}$ and $M_n$. 
The notation in equation~(\ref{eq:MnUpdate}) is suggestive of the line
of argument we will take in the following sections. To outline, we
may consider the singular value decomposition of the propagator
$B_{0:n}^T= V_n S_n U_n^T$, and decompose the error
covariances into a basis of the left singular vectors.  In particular,
we know that this decomposition may be written as a function of the
singular values, provided we have an appropriate bound on $M_n$ in
equation~(\ref{eq:MnUpdate}).  Moreover, the left singular vectors of
the propagator $B_{0:n}$ will become arbitrarily close to the
backwards Lyapunov vectors of the system.

The properties of $B_{0:n}$ are basically determined by the dynamical
system and are discussed in the next section, while those of $M_n$ are
commonly discussed in the context of control theory and are discussed
in Section~\ref{sec:control} where we prove a useful bound on its matrix norm in Lemma~\ref{lem:Mnbound}.

\subsection{Oseledet's theorem} \label{sec:oseledet}
Note that the boundedness condition on $A_n$ implies the bound $\parallel B_{0:n} \parallel \leq (c_A)^n, \forall n$. Then Oseledet's multiplicative ergodic theorem in \cite{Oseledets68} states that for each non-zero vector $\u \in \mathbb{R}^d$ the limit
\begin{equation*}
\mu = \lim\limits_{n \rightarrow \infty} \frac{1}{n} \log \frac{\parallel B_{0:n} \u \parallel }{\parallel \u \parallel}
\end{equation*}
exists and assumes up to $d$ distinct values $\mu_1 \geq \cdots \geq \mu_d$ which are called the Lyapunov exponents.  We will assume 
\begin{align}
0>\mu_{d_0+1}
\end{align}
so that exactly $d_0 < d$ of the Lyapunov exponents are non-negative. Further, defining the matrices
\begin{equation}
\label{eq:EbnEfn} E^b_n(m) \equiv \left[ B_{m-n:m}
(B_{m-n:m})^\ast\right]^{\frac{1}{2n}}, \quad
E^f_n(m) \equiv \left[ (B_{m:m+n})^\ast B_{m:m+n}
\right]^{\frac{1}{2n}},
\end{equation}
 Oseledet's theorem guarantees that the following limits exit, namely
\begin{align} 
\label{eq:EbnConvEb}
E^b(m) & \equiv\lim\limits_{n \rightarrow \infty} E^b_n(m), \\ 
\label{eq:EfnConvEf}
E^f(m) &\equiv\lim\limits_{n \rightarrow \infty} E^f_n(m).
\end{align}
The eigen-vectors of $E^b(m)$ and $E^f(m)$ represented as the column
vectors of $L^b(m) = [\l^b_1(m),\cdots,\l^b_d(m)]$ and $L^f(m) =
[\l^f_1(m),\cdots,\l^f_d(m)]$ respectively are defined as the
\emph{backward} and the \emph{forward} Lyapunov vectors at time
$m$ \cite{Legras96}. We note that the asymptotic results in later sections will essentially use the backward Lyapunov vectors $L^b(m)$.

The convergence of the individual matrix entries in equations~(\ref{eq:EbnConvEb}) and (\ref{eq:EfnConvEf}) guarantee the convergence of their characteristic polynomials---whose coefficients are well-defined functions of the matrix entries---the roots of which are the eigen-values. Therefore,
\begin{equation*}
 \lim\limits_{n\rightarrow \infty} \Lambda_{E^b_n(m)} = \Lambda_{E^b(m)}, \quad \lim\limits_{n\rightarrow \infty} \Lambda_{E^f_n(m)} = \Lambda_{E^f(m)}
\end{equation*}
where we recall that $\Lambda_Z$ is a diagonal matrix comprised of eigen-values of $Z$. Using the notations from Section~\ref{sec:notations} we additionally find
\begin{align*}
\parallel \lambda_j\left(E^b(m)\right) \v_j\left(E^b_n(m)\right) - E^b(m) \v_j\left(E^b_n(m)\right) \parallel \leq & \left|\lambda_j\left(E^b(m)\right) - \lambda_j\left(E^b_n(m)\right)\right| \\
&+ \parallel E^b_n(m) -  E^b(m) \parallel
\end{align*}
from which we can infer that 
\begin{equation*}
\lim\limits_{n \rightarrow \infty} \parallel \lambda_j\left(E^b(m)\right) \v_j\left(E^b_n(m)\right) - E^b(m) \v_j\left(E^b_n(m)\right) \parallel = 0
\end{equation*}
leading to $\lim\limits_{n\rightarrow\infty} V_{E^b_n(m)} = V_{E^b(m)} =L^b(m)$. Similarly, $\lim\limits_{n\rightarrow \infty} V_{E^f_n(m)} = V_{E^f(m)}=L^f(m)$.

Oseledets theorem also asserts the eigen-values of $E^b(m)$ or $E^f(m)$
do not depend on the initial time $m$, are the same for the
forward and backward matrices and relate to the Lyapunov exponents as 
\begin{equation}
\label{def:LyapunovExp}
\mu_j = \log(\lambda_j(E)), \quad j \in \{1,\cdots,d\}
\end{equation}
where we deliberately drop the index $m$ and the superscript $b$ or
$f$ on $E$. However, the forward and backward Lyapunov vectors are different
from each other and they also depend on the time $m$, i.e., $L^b(k) \ne L^b(m) \ne L^f(m) \ne L^f(k)$ for $k \ne m$.

Consider the singular-value decomposition $B_{0:n} \equiv U_{n} S_{n}
(V_n)^T$ so that under the canonical inner product
\begin{align*} E^f_n(0) = \left[(B_{0:n})^T B_{0:n}
\right]^\frac{1}{2n}= [V_n (S_{n})^2 (V_n)^T]^\frac{1}{2n} = V_n
(S_{n})^\frac{1}{n} (V_n)^T \,,
\end{align*}
implying $V_{E^f_n(0)}= V_n$ and
\begin{equation}
\label{eq:forwardLyapunov} \lim\limits_{n\rightarrow \infty} \parallel
\v_{j,n} - \l^f_j(0) \parallel = 0
\end{equation}
where $\v_{j,n}$ (and similarly $\u_{j,n}$ below) is the
$j^{\text{th}}$ column vector of $V_n$ (respectively $U_n$).  Likewise, we obtain
\begin{align*} E^b_{n}(n) = \left[B_{0:n}
(B_{0:n})^T\right]^\frac{1}{2n} = \left[U_n (S_n)^2 (U_n)^T
\right]^\frac{1}{2n}=U_n (S_n)^\frac{1}{n} (U_n)^T \,
\end{align*}
from which we can deduce that $V_{E^b_n(n)}= U_n$ and
\begin{equation}
\label{eq:backwardLyapunov} \lim\limits_{n\rightarrow
\infty} \parallel \u_{j,n} - \l^b_j(n) \parallel=0.
\end{equation}
We also infer that
\begin{equation}
\label{eq:eigsingval} \left(\sigma_j(B_{0:n})\right)^\frac{1}{n} =
\lambda_j(E^b_n(n))= \lambda_j(E^f_n(0)).
\end{equation}

\subsection{Controllability and observability for linear dynamics}
\label{sec:control}
The notions of observability and controllability are dual notions
within filtering problems. Roughly observability is the condition
that given sufficiently many observations, the initial state of the
system can be reconstructed by using a finite number of observations.
Similarly, controllability can be described as the ability to move the
system from any initial state to a desired state over a finite time
interval.  Formally stated:
\begin{mydef}\label{def:observeable}The system~(\ref{eq:nonautgen}) is defined to be
\emph{completely observable} if $\forall n \geq 1$,
\begin{equation}
\label{eq:observe} 
\det\left(\sum_{m=0}^{d-1} \left(B_{n:n+m}\right)^T H_{n+m}^T Q_{n+m}^{-1}H_{n+m} B_{n:n+m}\right) \neq 0
\end{equation}
and it is defined to be \emph{completely controllable} if $\forall n \geq 0$,
\begin{equation}
\label{eq:control} 
\det\left(\sum_{m=1}^{d} B_{n+m:n+d} F_{n+m} F_{n+m}^T \left(B_{n+m:n+d}\right)^T\right) \neq 0.
\end{equation}
In addition we describe the system as \emph{uniformly completely observable}  (respectively \emph{uniformly completely controllable}) if equation (\ref{eq:observe}) (respectively (\ref{eq:control})) is bounded from zero uniformly in $n$. 
\end{mydef}

We will assume that the system in equations~(\ref{eq:nonaut}) is uniformly completely
observable, i.e., the inequality~(\ref{eq:observe}) is uniformly bounded away from zero. Note however that this system \emph{cannot} be controllable since the
determinant in the equation~(\ref{eq:control}) is identically zero for a deterministic, perfect-model system as $F_n = 0, \forall n$.The hypothesis of uniform complete observability assures that the error covariance matrices remain bounded over time as seen below.
\begin{lemma}
\label{lem:boundedness}
Suppose that the linear, non-autonomous system~(\ref{eq:nonaut}) where the initial state $\x_0$ has a Gaussian law with mean $\x_{0\mid0}$ and covariance $\Delta_0$ is uniformly completely observable (Definition \ref{def:observeable}). Then the error covariance matrices remain bounded for all time, i.e, there exist constants $c_{\Sigma}$ and $c_{\Delta}$ such that $\forall n$, $\parallel \Delta_n \parallel \leq c_{\Delta}$ and $\parallel \Sigma_n \parallel \leq c_{\Sigma}$.
\end{lemma}
\begin{proof}
The result is proven for autonomous systems in Kumar~\cite{Kumar86}, Chapter~7, equations~($2.36$) and ($2.37$). Extension to the non-autonomous case is straightforward by rehashing the steps and changing the constants of the autonomous system to their time-varying counterparts.
\end{proof}

One should note the recent work of Ni et al. \cite{Ni2016} has demonstrated a stronger result: in continuous, perfect model systems the assumption of uniform complete observability is sufficient to demonstrate the stability of the Kalman filter.  In particular this shows that all solutions to the continuous Riccati equation for any choice of initial error covariance are bounded and converge to the same solution asymptotically.  This strongly suggest the same can be shown for the discrete time system, and we will return to this point in our discussion of results in Section \ref{sec:discussion}. 

Utilizing only the boundedness of the error covariance matrices, we demonstrate that the matrix $M_n$ stays bounded in the following lemma.
\begin{lemma}
\label{lem:Mnbound} 
Consider the uniformly completely observable, perfect-model, linear, non-autonomous system~(\ref{eq:nonaut}) where the initial state $\x_0$ has a Gaussian law with covariance $\Delta_0 >0$. Then the matrix $M_n$ defined in equation~(\ref{eq:Mn}) is uniformly bounded, i.e., there exist a constant $c_M$ such that $\parallel M_n \parallel \leq c_M, \forall n$.
\end{lemma}
\begin{proof}
We first show that the analysis error covariance matrix satisfies the recursive equation
\begin{equation}
\label{eq:analysisrecursive}
\Delta_n = (I-K_n H_n) A_n \Delta_{n-1} A_n^T (I-K_n H_n)^T + K_n Q_n K_n^T.
\end{equation}
Plugging in the Kalman update equations~(\ref{eq:Sigmaupdate}) and (\ref{eq:Sigmaupdate}), the R.H.S of equation~(\ref{eq:analysisrecursive}) equals $\Delta_n - (\Delta_n H_n^T - K_n Q_n) K_n^T$. The equation $(4.29)$ in \cite{Cohn97} establishes the equality $K_n = \Delta_n H_n^T Q_n^{-1}$ from which the recursion~(\ref{eq:analysisrecursive}) follows; further implying that
\begin{equation*}
\Delta_{n} \geq (I-K_n H_n) A_n \Delta_{n-1} A_n ^T (I-K_n H_n)^T.
 \end{equation*}
Recursively applying the above inequality gives $\Delta_{n} \geq M_n \Delta_0 M_n^T$. Decomposing $\Delta_0 = V_{\Delta_0}\Lambda_{\Delta_0}  V_{\Delta_0}^T$ and employing Lemma~\ref{lem:boundedness} we find
\begin{equation*}
\left \parallel M_n V_{\Delta_0} \Lambda_{\Delta_0}^{\frac{1}{2}}\right \parallel ^2 \leq \parallel \Delta_{n} \parallel \leq c_{\Delta}.
\end{equation*}
As $\parallel M_n \parallel \leq \left \parallel M_n V_{\Delta_0} \Lambda_{\Delta_0}^{\frac{1}{2}}\right \parallel \left \parallel \Lambda_{\Delta_0}^{-\frac{1}{2}} V_{\Delta_0}^T \right \parallel$ the result follows. Note that as $\Delta_0 >0$ the matrix $\Lambda_{\Delta_0}^{-\frac{1}{2}}$ is well-defined.
\end{proof}

\noindent
Bearing this bound in mind we shall proceed to discuss the asymptotic properties of the error covariance matrices.

\subsection{The asymptotic rank deficiency of the error covariance}
\label{sec:errcovrankdeficiency}
We begin by introducing a lemma which allows us to formally describe the collapse of the eigenvalues of the error covariance matrix. 
\begin{lemma}
\label{lemma:epsilonnull} For a given $\epsilon > 0$, let $Z\in\mathbb{R}^{d\times d}$ be a
symmetric matrix such that there is a $k \leq d$ dimensional
subspace $\W \subset \mathbb{R}^d$ for which
\begin{equation*} \sup\{\parallel Z \u \parallel
: \parallel\u\parallel =1, \u \in \mathcal{W} \} < \epsilon.
\end{equation*} Then $\dim\left( \E^{\epsilon}(Z) \right) \geq k$ where the subspace $\E^{\epsilon}$ is in accordance with Definition~\ref{def:alphaeigenspace}.
\end{lemma}
\begin{proof} Let $\{\v_1,\cdots,\v_d\}$ be an orthonormal eigenvector
basis for $Z$ corresponding to $\mid\lambda_1(Z)\mid
\geq \cdots \geq \mid \lambda_d(Z)\mid,$
and let $\{\u_1,\cdots,\u_k\}$ be a basis for
$\W$ of unit magnitude, such that we write
\begin{equation*} \u_l = \sum\limits_{j=1}^d \beta_{l,j} \v_j;
\hspace{10pt} l \in \{1,2,\ldots,k\}
\end{equation*} and the matrix of coefficients
\begin{equation*}
\begin{pmatrix} \beta_{1,1} & \beta_{1,2} & \cdots &
\beta_{1,d-k+1}&0&\cdots&0 \\ \beta_{2,1} & \beta_{2,2} & \cdots &
\beta_{2,d-k+1}& \beta_{2,d-k+2}&\cdots&0 \\ \vdots & \vdots & \vdots
& \vdots&\ddots&\vdots&\vdots \\ \beta_{k-1,1} & \beta_{k-1,2} &
\cdots & \cdots& \cdots&\beta_{k-1,d-1}&0 \\ \beta_{k,1} & \beta_{k,2}
& \cdots & \cdots& \cdots&\beta_{k,d-1}&\beta_{k,d}
\end{pmatrix}
\end{equation*} is in column echelon form where for every column index
$j > d-k+1$, the entries
\begin{equation*} \beta_{1,j}=\cdots=\beta_{k+j-d-1,j}=0
\end{equation*} and for every row index $l \leq k$,
$\sum\limits_{j=1}^{d-k+l} \beta_{l,j}^2 = 1$ corresponding $\parallel
u_{l}\parallel = 1$. Furthermore, as $Z$ is symmetric its
eigen-vectors form an orthonormal basis and hence $\parallel Z
\u_{l} \parallel^2 =\sum\limits_{j=1}^{d-k+l} \beta_{l,j}^2
\lambda_{j}^2(Z)$. For every $1\leq l \leq k$, setting $s=k-l+1$ we
find
\begin{equation*} \epsilon^2 > \parallel Z \u_{s} \parallel^2 =
\sum\limits_{j=1}^{d-k+s} \beta_{s,j}^2 \lambda_{j}^2(Z) \geq
\lambda_{d-k+s}^2(Z) = \lambda_{d-l+1}^2(Z).
\end{equation*} Hence the $k$ smallest eigen-values in absolute
magnitude satisfy
\begin{equation*} \mid \lambda_{d}(Z)\mid \leq \cdots \leq \mid
\lambda_{d-k+1}(Z) \mid < \epsilon
\end{equation*} and the result follows.
\end{proof}

\begin{theorem}
\label{thm:eigenvalueszero} 
Consider the uniformly completely observable, perfect-model, linear, non-autonomous system~(\ref{eq:nonaut}) where the initial state $\x_0$ has a Gaussian law with covariance $\Delta_0$. Then $\forall \epsilon>0$, $\exists n_1 > 0$
such that if $n \geq n_1$, $\Sigma_n$ and $\Delta_n$ will each have
at least $d-d_0$ eigen-values which are less than $\epsilon$ where
$d-d_0$ is the number of \textbf{negative} Lyapunov exponents of the
system~(\ref{eq:nonaut}), i.e.,
\begin{equation}\begin{matrix} \dim\left( \E^{\epsilon}(\Sigma_n)
\right) \geq d-d_0,& \text{and}& \dim\left( \E^{\epsilon}(\Delta_n)
\right) \geq d-d_0
\end{matrix} \label{dim-unstable}\end{equation}
where the subspace $\E^{\epsilon}$ is in accordance with Definition~\ref{def:alphaeigenspace}.
\end{theorem}
\begin{proof}
As denoted earlier, let $\mu_1 \geq \mu_2 \geq \cdots \geq \mu_d$ be the Lyapunov exponents of the system~(\ref{eq:nonaut}) where $d_0 < d$ of them are non-negative. The \emph{forward stable} Lyapunov vectors based at time zero is the set $\{\l^f_j(0)\}^{d}_{j=d_0+1}$ which by definitions~(\ref{eq:EbnConvEb}) and (\ref{def:LyapunovExp}) satisfy 
\begin{align}\label{eq:decay}
\lim_{n\rightarrow \infty} \frac{1}{n} \log\left(\left \parallel B_{0:n} \l^f_j(0) \right \parallel \right) = \mu_j.
\end{align}
Rewriting the analysis error covariance update equation in terms of the transpose
\begin{equation*}
\Delta_n =  M_n \Delta_0 B_{0:n}^T = B_{0:n}\Delta_0 M_n^T
\end{equation*}
we get $\Delta_n M_n^{-T} \Delta_0^{-1}  =  B_{0:n}$ and in particular 
\begin{equation*}
\Delta_n M_n^{-T} \Delta_0^{-1} \l^f_j(0) =  B_{0:n}\l^f_j(0).
\end{equation*}
Let us therefore define the sequence of vectors
\begin{align}
\w_{j,n} \equiv M_n^{-T} \Delta_0^{-1} \l^f_j(0).
\end{align}  
By Lemma \ref{lem:Mnbound} we know that $M_n$ is bounded above, so that the sequence of vectors $\w_{j,n} = M_n^{-T} \Delta_0^{-1} \l^f_j(0)$ must be bounded below.  As such, there is a constant $c_{\w}$ such that $c_{\w} \leq\parallel \w_{j,n} \parallel, \forall n $ and $j \in \{d_0+1,\ldots,d\}$. Choose a $\rho>0$ such that for each $j \in \{d_0+1 ,\ldots, d\}$, $\rho + \mu_j <0$.  Define $\overline{\w}_{j,n} \equiv \frac{\w_{j,n}}{\parallel \w_{j,n} \parallel} $. Then for a given $\epsilon > 0$, $\exists n_1$ such that for $n \geq n_1$
\begin{align}
\label{eq:forwardstablespace}
 \parallel \Delta_n  \overline{\w}_{j,n} \parallel = \frac{1}{\parallel \w_{j,n} \parallel}\parallel B_{0:n}\l^f_j(0)\parallel \leq \frac{1}{c_{\w}} e^{(\mu_j + \rho)n} < \epsilon.
\end{align} 
The theorem is therefore an immediate consequence of Lemma \ref{lemma:epsilonnull}. The proof for $\Sigma_n$ follows along similar lines.
\end{proof}

\subsection{Null space characterization and assimilation in the unstable subspace}
\label{sec:errcovnullspace}
The sequence of subspaces defined by the span of $\{\w_{j,n}\}_{j=d_0+1}^d$ will be the object of study for the remainder of this section.  In particular, we wish to establish the connection between this sequence of subspaces and assimilation in the unstable subspace which utilizes the \emph{backwards Lyapunov vectors}.  
\begin{mydef}
  \label{def:projections}
Define $\Lambda^s_{E^f_n(0)}$ to be the $d-d_0\times d-d_0$ diagonal matrix with diagonal entries given by $\left\{\lambda_j\left(E^f_n(0)\right)\right\}_{j=d_0 +1}^{d}$.
Also, let us define the following $d\times d- d_0$ operators 
\begin{align}
U^s_n = \left[\u_{d_0+1,n}, \cdots,
\u_{d,n}\right]\\
V^s_n = \left[\v_{d_0+1,n}, \cdots,
\v_{d,n}\right]\\
L^{bs}_n = \left[\l^b_{d_0+1}(n),\cdots,\l^b_{d}(n)\right]
\end{align}
Note that equation~(\ref{eq:backwardLyapunov}) implies that 
\begin{align}
\lim_{n\rightarrow\infty} \parallel U^s_{n} - L^{bs}_n\parallel=0.
\end{align}
\end{mydef}

Consider the equation~(\ref{eq:MnUpdate}), namely $\Delta_n =M_n
\Delta_0 V_n S_n U_n^T$, for the analysis error covariance $\Delta_n$
at time $n$ in terms of the matrix $M_n$ and the singular-value
decomposition of the propagator $B_{0:n}$.  Noting that
$B_{0:n}^T \u_{j,n} = \sigma_j(B_{0:n}) \v_{j,n}$ and
utilizing the relation~(\ref{eq:eigsingval}) we get
\begin{align}
\label{eq:projuj}
\Delta_nU^s_n \left(U^s_n\right)^T &= M_n \Delta_0
V^s_n\left(\Lambda^s_{E^f_n(0)}\right)^n\left(U^s_n\right)^T.
\end{align} 
Likewise, recalling that $\Sigma_{n} =
A_{n} \Delta_{n-1} A_{n}^T$, we can express the restriction of the
forecast error covariances as
\begin{align}
\label{eq:projujsigma}
\Sigma_nU^s_n
\left(U^s_n\right)^T &= A_nM_{n-1} \Delta_0
V^s_n\left(\Lambda^s_{E^f_n(0)}\right)^n\left(U^s_n\right)^T.
\end{align}

Making use of the above relations we now
prove one of our main result, which states that the norm of the
restriction of the analysis and forecast error covariances onto the
backwards stable Lyapunov subspaces must tend to zero.

\begin{theorem}
\label{thm:epsilonnull} 
Consider the uniformly completely observable, perfect-model, linear, non-autonomous system~(\ref{eq:nonaut}) where the initial state $\x_0$ has a Gaussian law with covariance $\Delta_0$. The restriction of $\Delta_n$ and $\Sigma_n$ into the span of the backwards stable Lyapunov vectors, $\{\l^b_j(n)\}_{j=d_0+1}^d$, tends to zero as $n\rightarrow \infty$.  That is
\begin{align}
\lim_{n\rightarrow \infty}\parallel \Delta_nL^{bs}_n\left(L^{bs}_n\right)^T\parallel = 0,\\
\lim_{n\rightarrow \infty}\parallel \Sigma_nL^{bs}_n\left(L^{bs}_n\right)^T\parallel = 0.
\end{align}   
\end{theorem}
\begin{proof} By definition $\log (\lambda_j(E^f(0))) = \mu_j$, so that the eigen-values
$\lambda_j(E^f(0))< 1$ correspond to the stable Lyapunov
exponents.  Recalling that $\lambda_{d_0+1}(E_n^f(0))\geq \cdots\geq
\lambda_{d}(E_n^f(0))$ we find $\left \parallel \Lambda^s_{E_n^f(0)} \right \parallel = \lambda_{d_0+1}(E_n^f(0))$ and 
\begin{equation}
\label{eq:operatornormless1} \lim\limits_{n\rightarrow
\infty}\left \parallel \Lambda^s_{E_n^f(0)} \right \parallel =
\lambda_{d_0+1}(E^f(0)) < 1.
\end{equation}
Consequent to equation~(\ref{eq:operatornormless1}) we can choose a small $0< \rho <1$ and sufficiently large $n_1$ such that when $n\geq n_1$, $\left \parallel \Lambda^s_{E_n^f(0)} \right \parallel \leq 1-\rho$.

The restriction of $\Delta_n$ into the span of the columns of $U^s_{n}$ is given by
the equation~(\ref{eq:projuj}).
Note the column vectors of $V^s_{n}$ and $U^s_{n}$ are
orthogonal and of unit norm, hence $\parallel V^s_{n} \parallel
= \parallel U^s_{n} \parallel = 1$. We then find for $n \geq n_1$
\begin{equation}\label{eq:projectdecay}
\parallel \Delta_nU^s_{n}\left(U^s_{n}\right)^T \parallel \leq \left \parallel
\Lambda^s_{E^f_n(0)} \right \parallel^n \parallel M_n \parallel \parallel
\Delta_0 \parallel \leq (1-\rho)^n c_M \parallel \Delta_0 \parallel.
\end{equation} 
Consider,
\begin{align}
\parallel \Delta_nL^{bs}_n\left(L^{bs}_n\right)^T\parallel &\leq \parallel \Delta_n\parallel \parallel L^{bs}_n\left(L^{bs}_n\right)^T - U^s_{n}\left(U^s_{n}\right)^T \parallel + \parallel \Delta_n U^s_{n}\left(U^s_{n}\right)^T\parallel,
\end{align}
and Lemma \ref{lem:boundedness} states $\parallel \Delta_n \parallel$ is bounded.  Therefore, 
\begin{align}
\lim_{n\rightarrow \infty}\parallel \Delta_nL^{bs}_n\left(L^{bs}_n\right)^T\parallel = 0
\end{align}
by equations~(\ref{eq:backwardLyapunov}) and (\ref{eq:projectdecay}).  This may be similarly stated for the forecast error covariance.
\end{proof}

The forecast and analysis error covariance matrices for a generic
non-autonomous system in general do not converge, but the above
results entail that asymptotically the only relevant directions for
the error covariance matrices are the backwards unstable-neutral Lyapunov directions validating the central hypothesis made by Trevisan et al. \cite{TrevisanPalatella2011} in their proposed reduced rank Kalman filtering algorithms.  

An intriguing consequence from equation (\ref{eq:forwardstablespace}) in Theorem \ref{thm:eigenvalueszero} is the following corollary.
\begin{corollary}
Suppose that for some $\epsilon_0>0$, $N_0 >0$, and for every $0< \epsilon <\epsilon_0$, $n>N_0$, 
\begin{align}
\dim\left(\mathcal{E}^\epsilon(\Delta_n)\right) = d-d_0,
\end{align}
ie: asymptotically the rank deficiency of the analysis error covariance $\Delta_n$ is exactly of dimension $d-d_0$. Then the transformation $M_n^{-T} \Delta_0^{-1}$  asymptotically maps the forwards stable vectors $\{\l^f_j(0)\}_{j=d_0+1}^d$ into the span of the backwards stable vectors $\{\l^b_j(n)\}_{j=d_0+1}^d$ as $n\rightarrow \infty$.
\end{corollary}

\subsection{Numerical results for a $30$-dimensional system}
\label{sec:errcovnumer}

Below we provide an illustration for this \emph{asymptotic rank deficiency} property of the error covariance matrices. The 
state space vector $\x_n$ and the observation vector $\y_n$
have dimension $d=30$ and $q=10$ respectively. This choice
is arbitrary and our simulations with different $d$ and $q$ have shown qualitatively equivalent results.

The time-varying, invertible propagators $A_n \in \mathbb{R}^{30
  \times 30}$, the observation error covariance matrices $Q_n \in
\mathbb{R}^{10 \times 10}$ and the observation matrices $H_n \in
\mathbb{R}^{10 \times 30}$ were all randomly generated for
sufficiently large $n$. We employed the $QR$ method \cite{Legras96} to
numerically compute the Lyapunov vectors and the Lyapunov
exponents and it was found that the number of non-negative Lyapunov exponents was $d_0 = 14$.
Starting from a random positive-definite $\Delta_0$, the sequence $(\Sigma_n, \Delta_n)$ was generated based on the
Kalman update equations~(\ref{eq:Sigmaupdate})-(\ref{eq:Deltaupdate}). For every $n$ we
computed the eigen-values of $\Delta_n$ sorted in descending order.

Figure~\ref{fig:eigenvaluesconvergence} shows the eigen-values of
$\Delta_n$ as a function of $n$. Barring the dominant $14$
eigen-values, the rest converge to zero serving as a visual testament
to Theorem~\ref{thm:eigenvalueszero}. Furthermore, we also calculated the
norm $\parallel \Delta_n \u_{j,n}\parallel, j \in \{1,2,\cdots,d\}$
for all $n$ and plot them in Figure~\ref{fig:normprojcoeff}. These
norm values are unsorted meaning that the topmost line in
Figure~\ref{fig:normprojcoeff} represent the values $\parallel
\Delta_n \u_{1,n}\parallel$ and the bottommost line denote $\parallel
\Delta_n \u_{d,n}\parallel$ for different values of $n$.  For $j > d_0
= 14$, $\parallel \Delta_n \u_{j,n} \parallel$ approaches zero
suggesting that as $n \rightarrow \infty$, the row space of $\Delta_n$
(and also $\Sigma_n$) coincides the space spanned by the
unstable-neutral, backward Lyapunov vectors, i.e., the bounds in
inequalities~\eqref{dim-unstable} are saturated.

\begin{figure}[!ht]
\begin{minipage}[c]{0.495\linewidth} \centering
	\includegraphics[width=1\textwidth]{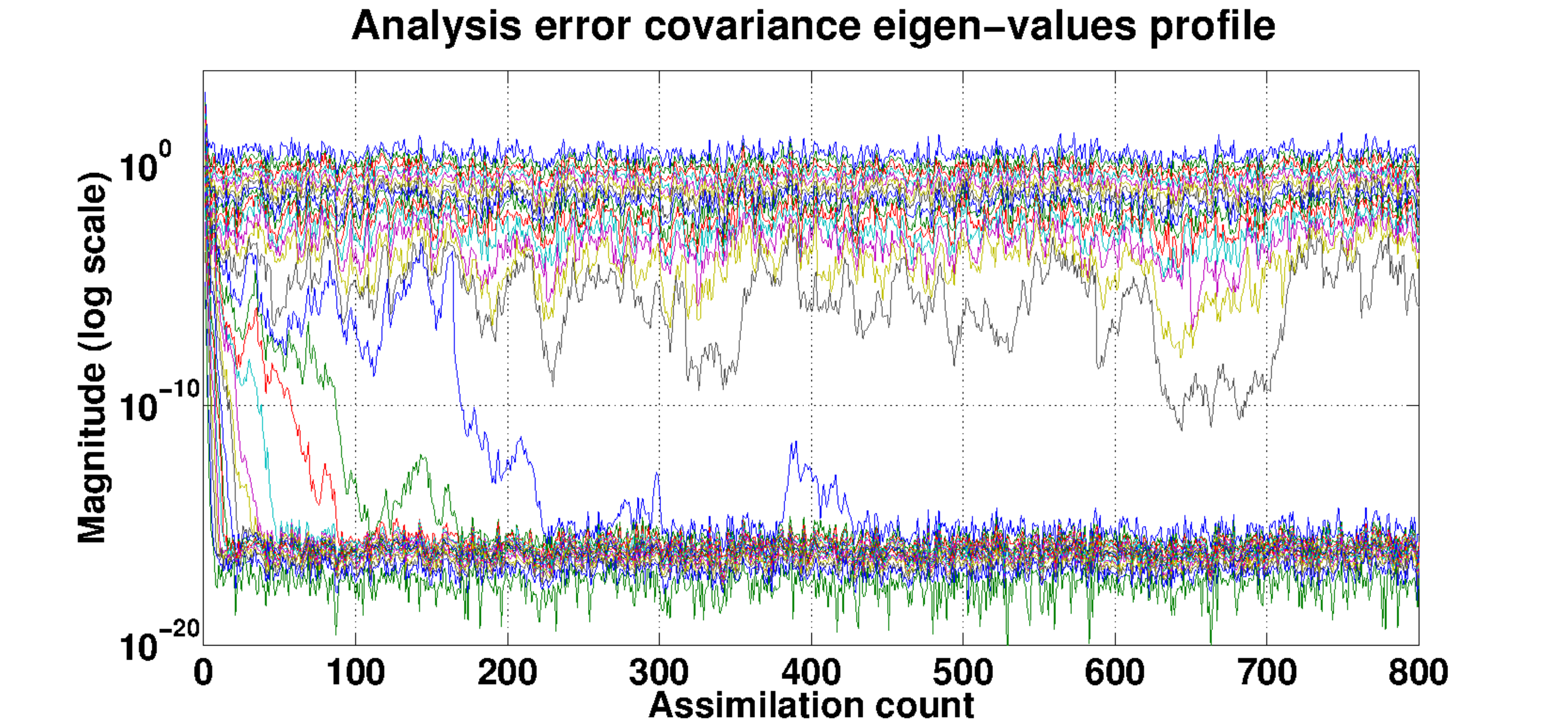}
	\caption{Profile of the eigen-values of $\Delta_n$. Counting establishes that the bottom $16$ eigen-values
converge to zero.}
	\label{fig:eigenvaluesconvergence}
\end{minipage}%
\begin{minipage}[c]{0.495\linewidth} \centering
	\includegraphics[width=1\textwidth]{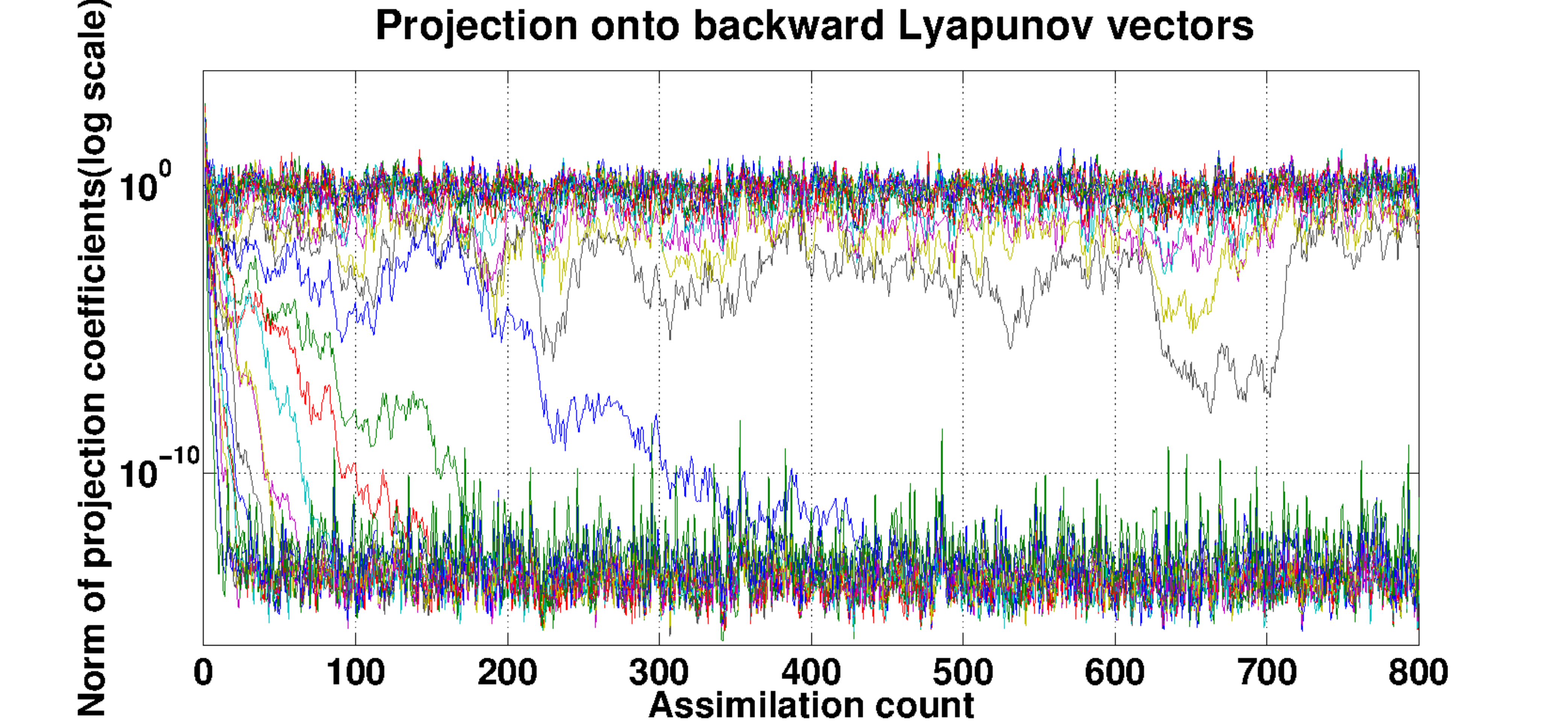}
	\caption{Norm of the projection \newline coefficients $\parallel \Delta_n \u_{j,n}\parallel$ for varying observation time $n$.}
	\label{fig:normprojcoeff}
\end{minipage}
\end{figure}

\section{Autonomous linear dynamical systems}
\label{sec:autosys}
\subsection{Null space characterization for autonomous systems}
\label{sec:autnull}
The noiseless, linear autonomous system can be defined from
equation~(\ref{eq:nonaut}), with the additional assumptions that $A_n
\equiv A$, $H_n \equiv H$, $Q_n \equiv Q$, are fixed matrices for all
$n$ --- therefore the results about the asymptotic rank deficiency property of the error covariance
matrices in Section~\ref{sec:nonautosys} also apply to autonomous systems.  However, a stronger statement can be made for time invariant systems because the backwards Lyapunov vectors will not vary in time.  In fact, the result in this section is even valid for the case
when only the dynamical system is autonomous ($A_n \equiv A$) but the
observation process is time dependent ($H_n$ and $Q_n$ depend on $n$).  

Akin to the non-autonomous case we define
\begin{equation}
\label{eq:EbnEfnauto} 
E^b_n \equiv \left[A^n (A^n)^\ast\right]^{\frac{1}{2n}}, \quad
 E^f_n \equiv\left[(A^n)^\ast A^n\right]^{\frac{1}{2n}}
\end{equation} 
and the similarity with equation~(\ref{eq:EbnEfn}) can readily be seen by setting $B_{m:m+n} = A^n,
\forall m$ in equation~(\ref{eq:Bn}) (hence the omission of the time index
$m$). As before, the existence of the limits
\begin{equation}
\label{eq:defEf}
E^b \equiv\lim\limits_{n \rightarrow \infty}E^b_n \,, \quad
E^f \equiv\lim\limits_{n \rightarrow \infty}E^f_n
\end{equation}
is guaranteed by Oseledets theorem \cite{Legras96}. The eigen-vectors
of $E^b$ and $E^f$ are called the backward and forward Lyapunov
vectors, represented here as the columns vectors of $L^b$ and $L^f$
ordered left to right from the most unstable direction---corresponding
to the largest Lyapunov exponent---to the most stable
direction---corresponding to the smallest Lyapunov exponent.
Specifically, the Lyapunov vectors are defined globally and have
\emph{no} dependence on the time in the linear, autonomous case.  Without the time dependence on the backwards stable Lyapunov vectors, we obtain a stronger statement about the asymptotic null space of the covariance matrices. 

\begin{mydef} Let $L^{bs} \equiv L^{bs}_n = \left[\l^b_{d_0+1},\cdots,\l^b_{d}\right]$. Note that Theorem~\ref{thm:eigspaceequality} proved in Appendix~\ref{sec:eigspaceLypVecAut} states tells us that the span of the columns of $L^{bs}$ is equal to $\E^1(A^T)$
\end{mydef}

\begin{corollary}
\label{thm:nullspaceS1AT}
Consider the uniformly completely observable, perfect-model, linear, autonomous system defined from equation~(\ref{eq:nonaut}) where $A_n\equiv A$, but $H_n$ and $Q_n$ may depend on $n$ and the initial state $\x_0$ has a Gaussian law with covariance $\Delta_0$. Then the restriction of the analysis and forecast error covariances onto $\E^1(A^T)$ tend to zero as $n\rightarrow\infty$.  That is
\begin{align}
\lim_{n\rightarrow \infty}\parallel \Delta_nL^{bs}\left(L^{bs}\right)^T\parallel = 0\\
\lim_{n\rightarrow \infty}\parallel \Sigma_nL^{bs}\left(L^{bs}\right)^T\parallel = 0
\end{align}   
\end{corollary}
\begin{proof}
Combining Theorem~\ref{thm:epsilonnull} with Theorem~\ref{thm:eigspaceequality} this is a straightforward consequence.
\end{proof}

In all our numerical simulations with arbitrary (and completely observable) choices of $A$, $H$ and $Q$ we have additionally observed \emph{convergence} of  $\Delta_n$ and $\Sigma_n$ to a fixed $\Delta$ and $\Sigma$ respectively and seen their null spaces contain $\E^1\left(A^T\right)$ as stated by Corollary~\ref{thm:nullspaceS1AT} (refer Section~\ref{sec:autnumeric}).  Considering the recent work of Ni et al. \cite{Ni2016}, this strongly suggests that the classical result of the stable Riccati equation for completely observable and controllable, discrete autonomous systems \cite{Kumar86} has an analogue in the case of completely observable, perfect model systems.

\subsection{Numerical results for linear autonomous system} 
\label{sec:autnumeric}
We choose a non-singular matrix $A \in \mathbb{R}^{30 \times 30}$ ($d = 30$)
consisting of random entries and set $d_0 =12$ of its eigen-values to
be greater or equal to one in absolute magnitude. We ran the Kalman filtering
system long enough and observed that the analysis error
covariance do converge to a fixed $\Delta$ and then projected $\Delta$ onto
the generalized eigen-space of
$A^T$. Figure~\ref{fig:autsys_eigenValuesAandLyp} plots the absolute
magnitude of eigen-values of $A$ sorted in descending order
$\left(\mid\lambda_1(A)\mid \geq \cdots \geq \mid \lambda_d(A)\mid
\right)$ in blue color and shows the Lyapunov
exponents for this system in red shade where we note that the number of non-negative
Lyapunov exponents is exactly $12$ tantamount to the number of
eigen-values of $A$ greater than or equal to one in magnitude. Additionally, it
can be verified that the Lyapunov exponents are just the logarithm (to
the base $e$) of the absolute magnitude eigen-values of $A$. Recalling
the definition of the Lyapunov exponents from equation~(\ref{def:LyapunovExp}),
this equality also lends credence to our
Theorem~\ref{thm:eigvalcon}. The plot in
Figure~\ref{fig:autsys_normprojcoeff} displays $\parallel \Delta
\left(\v_j\left(A^T\right)\right)\parallel; j \in \{1,2,\cdots,d\}$
where $\v_j\left(A^T\right)$ is the generalized eigen-vector of
$\lambda_j(A)$. Observe that when $j > 12$, the norm of the projected
coefficients is zero rendering a visual confirmation to
Corollary~\ref{thm:nullspaceS1AT}. 

\begin{figure}[!ht]
\begin{minipage}[c]{0.495\linewidth} \centering
	\includegraphics[width=1\textwidth]{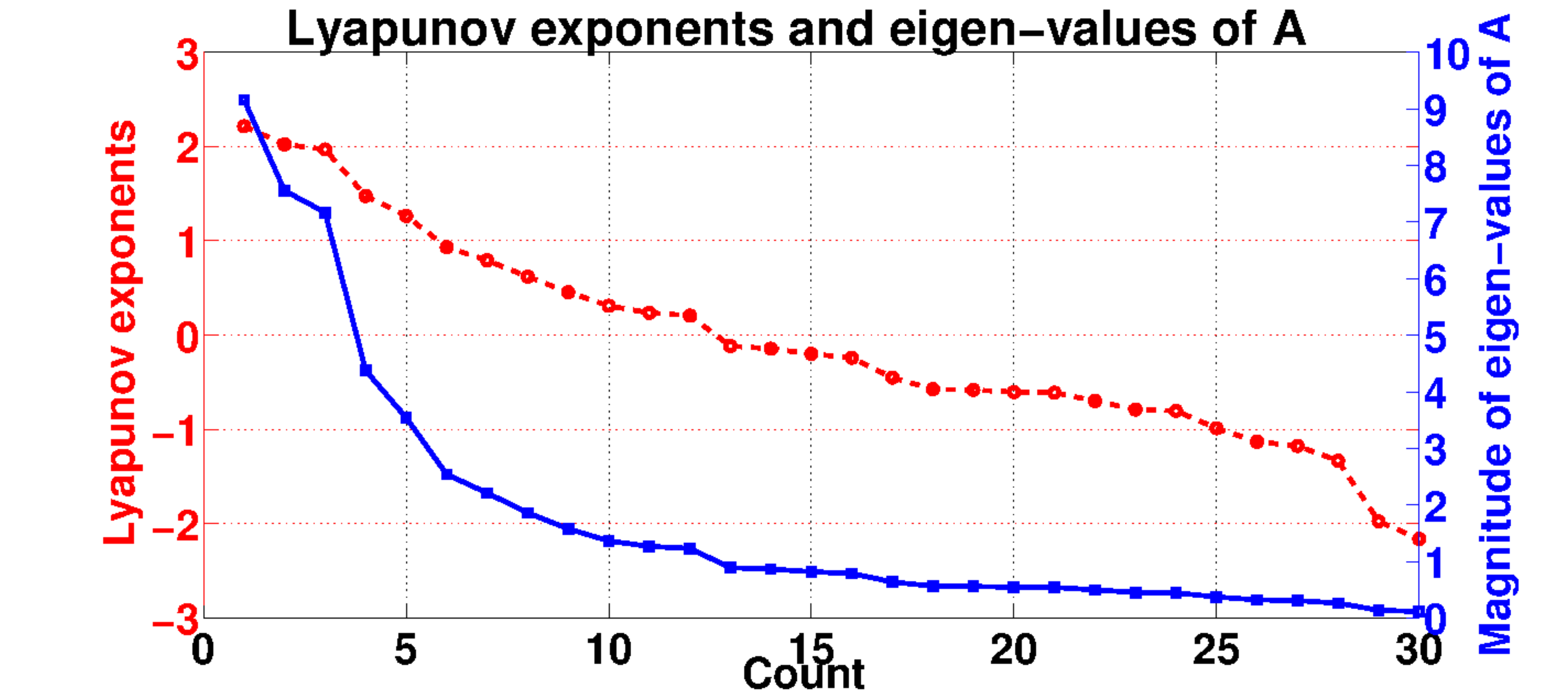}
	\caption{Lyapunov exponents in blue and the magnitude of the eigen-values of $A$ in red.}
	\label{fig:autsys_eigenValuesAandLyp}
\end{minipage}%
\begin{minipage}[c]{0.495\linewidth} \centering
	\includegraphics[width=1\textwidth]{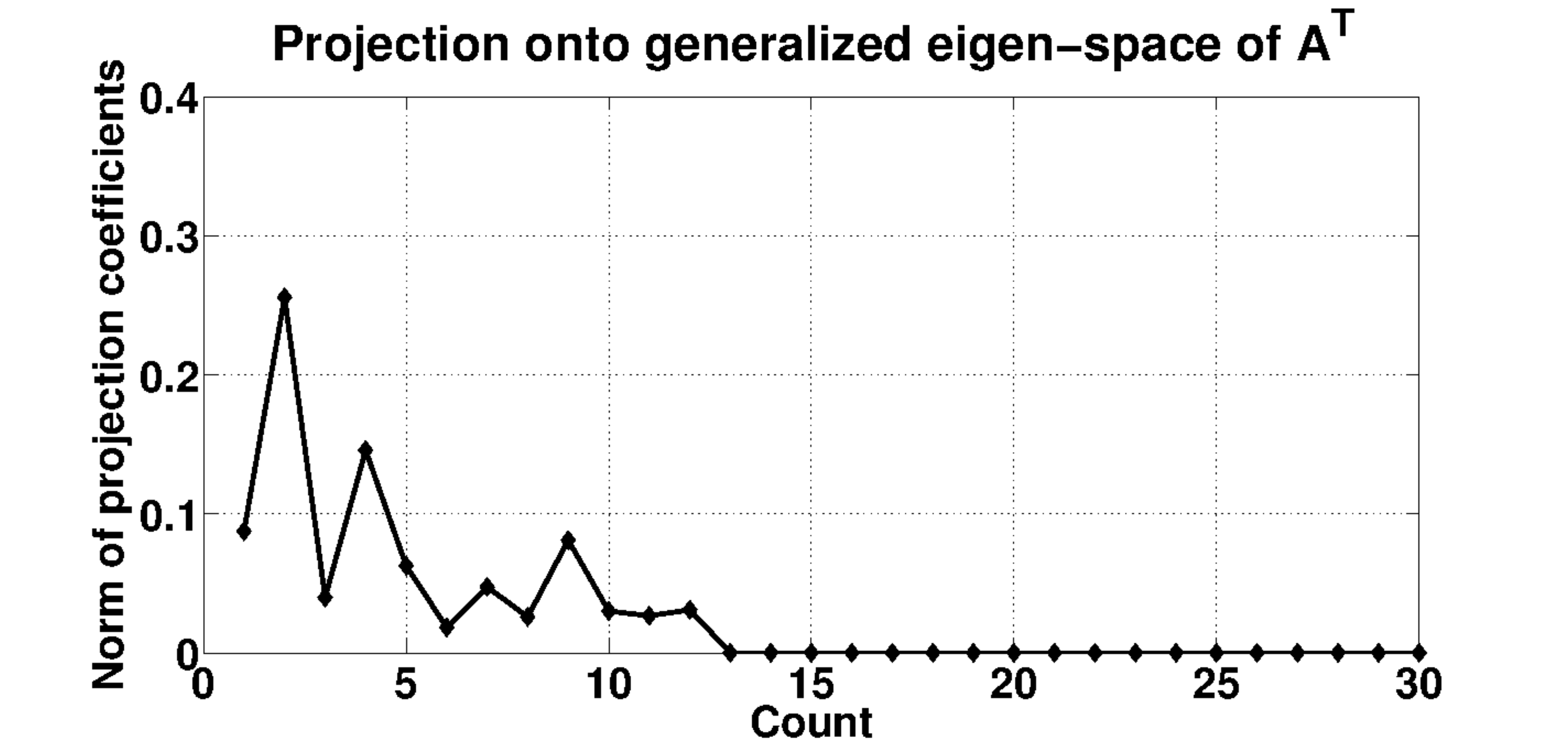}
	\caption{Norm of the projection coefficients onto the generalized
eigen-space of $A^T$.}
	\label{fig:autsys_normprojcoeff}
\end{minipage}
\end{figure}

\section{Discussion}
\label{sec:discussion}
We have shown that under sequential Kalman filtering, the error covariance for a linear, perfect-model, conditionally Gaussian
systems asymptotically collapses to the subspaces spanned by the
backwards unstable Lyapunov vectors. This has been known to practitioners in the forecasting community
\cite{Bonnabel13}, but had yet to be stated in precise mathematical terms. In particular, 
this foundational work validates the underlying assumptions and methodology of AUS.

At the same time, these results open many new questions for ongoing research
related to AUS algorithms. For instance, the present results do not formally show the equivalence of a fully reduced-rank algorithm such
as EKF-AUS applied in such a setting. The conditions that imply the
convergence of the covariance matrices, given arbitrary
low rank symmetric matrices chosen as initial conditions 
have yet to be established.  Recent work strongly suggests that filter stability for discrete, perfect model systems can be demonstrated under sufficient observability hypotheses alone \cite{Ni2016}.  Determining the necessary hypotheses for stability of the discrete Kalman filter with low rank initializations of the prior covariance matrix in perfect model systems will be the subject of the sequel to our work. 

Additionally there are conceptual issues to be resolved in bridging
the results for linear systems to non-linear settings; the
former having the advantage of Lyapunov vectors being defined
globally in space, whereas the formulation must change in a non-linear
setting, respecting the dependence on the underlying path. Both of these
directions of inquiry open rich areas for mathematical research and
future algorithm design.

While the ultimate goal of DA is a precise estimate of
state for chaotic dynamics, it is critical to understand the uncertainty of the prediction. An exact calculation of
the posterior distribution of states for a high dimensional, complex
system is computationally intractable; as computational resources
increase, so will model complexity and thus computational efficiency
alone will not resolve this issue. This work provides an idealized,
but general framework for future investigations into low dimensional
approximations for uncertainty calculation. We hope that a precise mathematical framework for
understanding the nature of uncertainty for linear systems will lead
to innovative research to surmount these challenges.
\appendix
\section{Eigen-values, singular-values and Lyapunov exponents of
linear autonomous systems} 
The results established in this Appendix and
Appendix~\ref{sec:eigspaceLypVecAut} should be treated as an
independent body of work elucidating the relationship between various
concepts in linear, autonomous systems and not restricted to the
domain of DA and filtering theory. While these
relationships are known and can be retrieved from multiple sources in
the literature, we have explicitly proved them here for
completeness. Readers familiar with these mathematical connections may
choose to skip through these sections without any loss of continuity.

\label{sec:eigvalSigvalAut} Based on the definition of the matrix
$E^f_n$ in equation~(\ref{eq:EbnEfnauto}) we find $\lambda_j(E^f_n)=
\left[\sigma_j(A^n)\right]^{\frac{1}{n}}$. As $E^f_n \rightarrow E^f$
we also have
\begin{equation}\begin{matrix}
\label{eq:eigvallimit} \lim\limits_{n \rightarrow \infty}
\lambda_j(E^f_n) = \lim\limits_{n \rightarrow \infty}
\left[\sigma_j(A^n)\right]^{\frac{1}{n}} = \lambda_j(E^f) & & j
\in\{1,2,\cdots,d\}
\end{matrix}\end{equation} where the eigen-values $\lambda_j$ and
singular-values $\sigma_j$ are ordered descending in norm.  Dropping
the label for brevity let $J = V_A^{-1} A V_A$ (instead of $J(A)$) be
the Jordan canonical form of $A$. It is straightforward to see that
$A^n = V_A J^n V_A^{-1}$ for any integer $n$. The following inequality stated in Theorem~9 of \cite{Merikoshi04} is
quite useful. For any two square matrices $Z_1$ and $Z_2$ we have
\begin{equation}\begin{matrix}
\label{eq:inequality1} \sigma_j(Z_1) \sigma_d(Z_2) &\leq& \sigma_j(Z_1
Z_2) &\leq& \sigma_j(Z_1) \sigma_1(Z_2)
\end{matrix}\end{equation} Since the singular-values of both the
matrix and its transpose are the same, it follows that
\begin{equation}\begin{matrix}
\label{eq:inequality2} \sigma_d(Z_1) \sigma_j(Z_2) &\leq& \sigma_j(Z_1
Z_2) &\leq& \sigma_1(Z_1) \sigma_j(Z_2).
\end{matrix}\end{equation}
\begin{lemma}
\label{lemma:Jordansigrelation} For any square matrix $Z = V_Z J(Z) V_Z^{-1}$
\begin{equation*} \lim\limits_{n \rightarrow \infty}
\left[\sigma_j(Z^n)\right]^{\frac{1}{n}} = \lim\limits_{n \rightarrow
\infty} \left[\sigma_j(J(Z)^n)\right]^{\frac{1}{n}}
\end{equation*}
\end{lemma}
\begin{proof} Inequalities~(\ref{eq:inequality1}) and (\ref{eq:inequality2}) leads to
\begin{equation*} 
 \sigma_d\left(V_Z\right) \sigma_d\left(V_Z^{-1}\right) \sigma_j(J(Z)^n) \leq \sigma_j(Z^n) \leq \sigma_1\left(V_Z\right)
\sigma_1\left(V_Z^{-1}\right) \sigma_j(J(Z)^n).
\end{equation*} 
Raising each term to the power $1/n$ and letting $n\rightarrow \infty$ proves the result.
\end{proof}
\begin{corollary}
\label{coro:EJrelation} For any matrix $A$ let $E^f$ be defined as in equation~(\ref{eq:defEf}) and $J$ be the Jordan canonical form of $A$. Then
$\lambda_j(E^f) = \lim\limits_{n\rightarrow \infty} \left[\sigma_j(J^n)\right]^{\frac{1}{n}}, j \in\{1,2,\cdots,d\}.$
\end{corollary}
\begin{proof} The results follows immediately when we employ
Lemma~\ref{lemma:Jordansigrelation} setting $Z=A$ in conjunction with equation~(\ref{eq:eigvallimit}).
\end{proof}

The theorem below establishes the relation between the eigen-values of
the time invariant propagator $A$ and the limit matrix $E^f$.
\begin{theorem}
\label{thm:eigvalcon} For any matrix $A$ let the matrix $E^f$ be
defined as in equation~(\ref{eq:defEf}). Then the eigen-values of $E^f$ equal
the absolute magnitude eigen-values of $A$, i.e, $\lambda_j(E^f) =
\left|\lambda_j(A)\right|, j \in\{1,2,\ldots,d\}$.
\end{theorem}
\begin{proof} We consider two different cases.

\noindent \textbf{case 1: $A$ is diagonalizable.} When $J$ is diagonal
then $\sigma_j(J) = \left|\lambda_j(J)\right| =
\left|\lambda_j(A)\right|$. Recalling that $\lambda_j(J^n) =
\left[\lambda_j(J)\right]^n, \forall n$, we get
$\left[\sigma_j(J^n)\right]^{\frac{1}{n}} = \left|\lambda_j(A)\right|$
and the result follows from Corollary~\ref{coro:EJrelation}.

\noindent \textbf{case 2: $A$ is not diagonalizable.} Let
$J_{\lambda}(A)$ denote the Jordan-block of size $k \times k$
corresponding to an eigen-value $\lambda$ of $A$ of the form
\begin{equation}
\label{eq:Jordanblock} J_{\lambda}(A) \equiv
\begin{pmatrix} \lambda & 1&0&\cdots&0 \\ 0&\lambda&1&\cdots&0\\
\vdots&\vdots&\vdots&\vdots&\vdots\\ 0&0&0&\lambda&1\\ 0&0&0&0&\lambda
\end{pmatrix}.
\end{equation}
The following lemma is useful in proving Theorem~\ref{thm:eigvalcon}.
\begin{lemma}
\label{lemma:singulareigenvalueequality} For any matrix $A$ let
$J_{\lambda}(A)$ be a Jordan block corresponding to eigen-value
$\lambda$ of $A$ as defined in equation~(\ref{eq:Jordanblock}). Then the
singular-values of $J_{\lambda}(A)$ respect the following equality,
namely
\begin{equation}
\label{eq:singulareigenvalueequality}
\begin{matrix} \lim\limits_{n \rightarrow \infty}
\left[\sigma_j\left(J_{\lambda}^n\right)\right]^{\frac{1}{n}} =
|\lambda|& & j \in\{1,2,\cdots,k\},
\end{matrix}\end{equation} i.e, the \emph{limiting singular-values are
the absolute magnitude of their respective eigen-values}.
\end{lemma}
\begin{proof}
Following the standard proof technique for equality
results we individually show that
\begin{equation}
\label{eq:lessthaninequality}
\begin{matrix} \lim\limits_{n \rightarrow \infty}
\left[\sigma_j\left(J_{\lambda}^n\right)\right]^{\frac{1}{n}} \leq
|\lambda|& & j \in\{1,2,\cdots,k\}
\end{matrix}
\end{equation} and
\begin{equation}
\label{eq:greaterthaninequality1}
\begin{matrix} \lim\limits_{n \rightarrow \infty}
\left[\sigma_j\left(J_{\lambda}^n\right)\right]^{\frac{1}{n}} \geq
|\lambda|& & j \in\{1,2,\cdots,k\}.
\end{matrix}
\end{equation}

\noindent Let the Nilponent matrix $N \equiv J_{\lambda} - \lambda I$
with $N^k = \mathbf{0}$. When $n \geq k-1$ we get
\begin{equation*}\begin{matrix} J_{\lambda}^n &=& \left(\lambda I +
N\right)^n &=& \sum\limits_{r=0}^{k-1} \binom{n}{r} \lambda^{n-r}
N^r.
\end{matrix}\end{equation*} Further, the highest singular-value
$\sigma_1(N^r) = 1$ for $r \in \{0,1,\cdots,k-1\}$. If $\lambda = 0$
then $J_{\lambda}^n = \mathbf{0}$ when $n \geq k-1$ and the result is
trivially true. Suppose $\lambda \not=0$ define $\delta \equiv
\frac{1}{\lambda}$. Using the identity that for any two matrices $Z_1$
and $Z_2$, $\sigma_1(Z_1+Z_2) \leq \sigma_1(Z_1) + \sigma_1(Z_2)$
as stated in Theorem~6 of \cite{Merikoshi04}, we have
\begin{equation} \sigma_1\left(J_{\lambda}^n\right) \leq |\lambda|^n
\left[\sum_{r=0}^{k-1} \binom{n}{r} |\delta|^r \right].
\end{equation} Let $|\delta| = \epsilon \xi$ for any $0 < \epsilon
\leq |\delta|$. Then
\begin{align*} \sigma_1\left(J_{\lambda}^n\right) & \leq |\lambda|^n
\xi^k \left[\sum_{r=0}^{k-1} \binom{n}{r} \epsilon^r \right] \nonumber
\\ & \leq |\lambda|^n \xi^k \left[\sum_{r=0}^{n} \binom{n}{r}
\epsilon^r \right] = |\lambda|^n \xi^k (1+\epsilon)^n
\end{align*} Raising to the power $1/n$ and taking the limit we get
\begin{equation*} \lim_{n \rightarrow \infty}
\left[\sigma_1\left(J_{\lambda}^n\right)\right]^{\frac{1}{n}} \leq
|\lambda| (1+\epsilon).
\end{equation*} The above inequality is also true for the rest of the
singular-values as $\sigma_1(.)$ is the largest. Since $\epsilon$ is
\emph{arbitrary} the first inequality~(\ref{eq:lessthaninequality})
follows. If $\lambda=0$ we get the desired, stronger equality result in equation~(\ref{eq:singulareigenvalueequality}) as the singular-values by
definition are non-negative. It suffices to focus on the case $\lambda
\not=0$ where $J_{\lambda}$ is invertible.

\noindent To establish the reverse inequality~(\ref{eq:greaterthaninequality1}), let $T_{\lambda}$ be the Jordan
canonical form of $J_{\lambda}^{-1}$ given by
\[ T_{\lambda} \equiv
\begin{pmatrix} \frac{1}{\lambda} & 1&0&\cdots&0 \\
0&\frac{1}{\lambda}&1&\cdots&0\\ \vdots&\vdots&\vdots&\vdots&\vdots\\
0&0&0&\frac{1}{\lambda}&1\\ 0&0&0&0&\frac{1}{\lambda}
\end{pmatrix}.
\] Lemma~\ref{lemma:Jordansigrelation} entails that
\begin{equation*} \lim\limits_{n \rightarrow \infty}
\left[\sigma_j\left(\left(J_{\lambda}^{-1}\right)^n\right)\right]^{\frac{1}{n}}
= \lim\limits_{n \rightarrow \infty}
\left[\sigma_j(T_{\lambda}^n)\right]^{\frac{1}{n}}.
\end{equation*} Applying the inequality~(\ref{eq:lessthaninequality}) on
$T_{\lambda}$ gives us
\begin{equation*}\begin{matrix} \lim\limits_{n \rightarrow \infty}
\left[\sigma_j\left(T_{\lambda}^n\right)\right]^{\frac{1}{n}} \leq
\frac{1}{|\lambda|}& & j \in\{1,2,\cdots,k\}.
\end{matrix}\end{equation*} In particular,
\begin{equation*} 
\lim_{n \rightarrow \infty}
\left[\sigma_1\left(\left(J_{\lambda}^{-1}\right)^n\right)\right]^{\frac{1}{n}}
= \lim\limits_{n \rightarrow \infty}
\frac{1}{\left[\sigma_k\left(J_{\lambda}^n\right)\right]^{\frac{1}{n}}}\leq
\frac{1}{|\lambda|}
\end{equation*} 
where the equality stems from the fact that for any
invertible matrix $Z$ of size $k \times k$
\begin{equation*} 
\sigma_j\left(Z^{-1}\right) =
\frac{1}{\sigma_{k-j+1}\left(Z\right)}.
\end{equation*} We then get
\begin{equation}
\label{eq:greaterthaninequality2} \lim\limits_{n \rightarrow \infty}
\left[\sigma_k(J_{\lambda}^n)\right]^{\frac{1}{n}} \geq |\lambda|.
\end{equation} Since $\sigma_k(.)$ is the smallest singular-value the
inequality~(\ref{eq:greaterthaninequality2}) is also valid for the
rest.
\end{proof}

\noindent Now to prove Theorem~\ref{thm:eigvalcon} note that for any $n$
\[ J^n = \begin{pmatrix} J_{\lambda_1}^n&0&\cdots&0\\
0&J_{\lambda_2}^n&\cdots&0\\ \vdots&\vdots&\ddots&\vdots\\
0&0&\cdots&J_{\lambda_l}^n
\end{pmatrix}
\] is a block diagonal matrix and the eigen-(singular) values of $J^n$
equals the \emph{disjoint union} of eigen-(singular) values of
individual Jordan blocks $J_{\lambda_1}^n, \cdots,
J_{\lambda_l}^n$. In accordance with Corollary~\ref{coro:EJrelation}
and Lemma~\ref{lemma:singulareigenvalueequality} we find $\forall j
\in \{1,2,\ldots,d\}$, 
\begin{equation*}
\lambda_j(E^f) = \lim\limits_{n \rightarrow
\infty} \left[\sigma_j(J^n)\right]^{\frac{1}{n}} =
\left|\lambda_j(J)\right| = \left|\lambda_j(A)\right|.
\end{equation*}
\end{proof}

\section{Eigen-spaces and Lyapunov vectors of linear autonomous
systems}
\label{sec:eigspaceLypVecAut} By a suitable coordinate transformation,
namely $\z_{n} = V_A^{-1} \x_{n}$, studying the dynamics $\x_{n+1} = A
\x_{n}$ is tantamount to investigating $\z_{n+1} = J \z_{n}$ where $J
= V_A^{-1}A V_A$ is the Jordan canonical form of $A$.  Indeed,
\begin{align*} \z_{n+1} &= J \z_n = V_A^{-1} A V_A V_A^{-1} \x_n =
V_A^{-1} \x_{n+1}.
\end{align*} Corresponding to the definitions of the matrices $E^f_n$
and $E^f$ in equations~\eqref{eq:EbnEfnauto}-\eqref{eq:defEf}, let
$G_n \equiv \left[ (J^n)^\ast J^n \right]^{\frac{1}{2n}}$ and let $G
\equiv \lim\limits_{n \rightarrow \infty} G_n$.

We consider the two systems in the different $d$ dimensional spaces
$\RA$ and $\CJ$ where the underlying propagators are $A$ and $J$
respectively. Note that as the matrix $V_A$ might be complex (though
$A$ is real) the dynamics for the propagator $J$ is examined in a
complex state space.
\begin{lemma}
\label{lem:idendityeigenvector} If the scalar product in $\CJ$ is the
canonical one namely, $\langle \u,\v \rangle_{J} = \u^{\dagger} \v$,
then $V_G = I_d$ where $I_d$ is the $d \times d$ identity matrix.
\end{lemma}
\begin{proof}
We find it convenient to handle the following scenarios
separately.

\noindent \textbf{case 1: $A$ is diagonalizable.} $J$ is diagonal and so is $J^n$.  In the canonical inner product setting the entries of the diagonal
$G_n$ are the absolute magnitude entries of $J$. It follows that $G$
is diagonal and $V_G = V_J = I_d$.

\noindent \textbf{case 2: $A$ is not diagonalizable.} As before,
consider the Jordan block $J_{\lambda}$ given in equation~(\ref{eq:Jordanblock}) of size $k \times k$ corresponding to the
eigen-value $\lambda$. Define $G_{\lambda} \equiv \lim\limits_{n
\rightarrow \infty} \left[(J_{\lambda}^n)^{\ast}
J_{\lambda}^n\right]^{\frac{1}{2n}}$. Since $G_{\lambda}$ is symmetric
it is diagonalizable and by Theorem~\ref{thm:eigvalcon} we have
$\lambda_j\left(G_{\lambda}\right) = |\lambda|, \forall j \in
\{1,2,\cdots,k\}$. As all the eigen-values of $G_{\lambda}$ are
equal, it is a scalar matrix and therefore
we can choose $V_{G_{\lambda}} = I_k$. Since
\[ G = \begin{pmatrix} 
G_{\lambda_1}&0&\cdots&0\\
0&G_{\lambda_2}&\cdots&0\\ \vdots&\vdots&\ddots&\vdots\\
0&0&\cdots&G_{\lambda_l}
\end{pmatrix}
\] the result follows.
\end{proof}
\begin{lemma}
\label{lem:eigenvectorrelation} Under the definition of the scalar
products $\langle \u,\v \rangle_{J} = \u^{\dag} V_A^{\dagger} V_A \v$
in $\CJ$ and $\langle \u,\v \rangle_{A} = \u^T \v$ in $\RA$, $V_G =
V_A^{-1} V_{E^f}$.
\end{lemma}
\begin{proof} For the aforesaid considerations of the scalar products
in $\CJ$ and $\RA$, $J^{\ast} = \left(V_A^{\dag} V_A \right)^{-1}
J^{\dagger} V_A^{\dag}V_A$ and $A^{\ast} = A^T$ respectively. Recalling that $J = V_A^{-1}A
V_A$ we have
\begin{align*} 
(J^{n})^{\ast} &= \left( V_A^{\dag} V_A \right)^{-1}
V_A^\dag \left(A^n\right)^T \left(V_A^{-1} \right)^\dag V_A^{\dag}V_A
=V_A^{-1} \left(A^n\right)^T V_A \\
\Rightarrow
G_n &= \left[V_A^{-1} \left(A^n\right)^T V_A
V_A^{-1}A^n V_A\right]^{\frac{1}{2n}} = \left[V_A^{-1}
\left(A^n\right)^T A^n V_A\right]^{\frac{1}{2n}}.
\end{align*}
As
$\left(E^f_n\right)^{2n} = \left(A^n\right)^T A^n$ is symmetric, it is
diagonalizable by an orthonormal matrix $V_{E^f_n}$ and carries a
representation $\left(E^f_n\right)^{2n}= V_{E^f_n}
(\Lambda_{E^f_n})^{2n} V_{E^f_n}^T$. We find $ \Lambda_{G_n} =
\Lambda_{E_n}$ and $V_{G_n} = V_A^{-1}V_{E_n},\forall n$ and the
result follows by letting $n \rightarrow \infty$.
\end{proof}

Recall the real span $\T_\w$ from Definition~\ref{def:realspan} bearing in mind the complex generalized
eigenvectors of any matrix $Z$ always occur in conjugate pairs $\{\w,
\overline{\w} \}$ with $\T_{\w} = \T_{\overline{\w}}$. We have the
following theorem, namely
\begin{theorem}[Eigenspace equality]
\label{thm:eigspaceequality} For any matrix $A$ let the matrix $E^f$
be defined as in equation~(\ref{eq:defEf}). Then for any $\alpha \geq 0$ the
corresponding $\alpha$-eigenspaces of $E^f$ and $A$ are the same, i.e,
$\Ea\left(E^f\right) = \Ea(A)$. Equivalently, $\Ea\left(E^b\right) = \Ea\left(A^T\right)$.
\end{theorem}

\begin{proof} By Theorem~\ref{thm:eigvalcon} we have $\lambda_j(G) =
|\lambda_j(J)| = |\lambda_j(A)| = \lambda_j(E^f)$. Recall that the
eigen-values are ordered with $\lambda_1(G)$ and $\lambda_d(G)$ being
the largest and the smallest respectively. Oseledets theorem states
that there exits a sequence of embedded subspaces 
\begin{equation*}
0 \subset \F_d \subset \F_{d-1} \subset \cdots \subset \F_1 = \CJ
\end{equation*}
such that on the complement $\F_j \backslash \F_{j+1}$
of $\F_{j+1}$ in $\F_j$ the growth rate is at most $\lambda_j(G)$
\cite{Oseledets68}. The subspaces $\F_j$ can be obtained as the direct sum
of the eigenvectors $\v_j(G)$ as 
\begin{equation*}
\F_j = \v_d(G) \oplus \v_{d-1}(G) \oplus \cdots \oplus \v_j(G)
 \end{equation*}
where $\v_j(G)$ is the eigenvector of $G$ corresponding
to $\lambda_j(G)$. Further, though the eigenvectors of $G$ depend on
the underlying scalar product in $\CJ$, the embedded subspaces $\F_j$
and the eigen-values $\lambda_j(G)$ are \emph{independent} of it
\cite{Legras96}. 

Corresponding to the two inner-product definitions in $\CJ$,
specifically $\langle \u,\v \rangle_{J} =\u^{\dagger} \v$ and $ \langle \u,\v \rangle_{J} = \u^{\dagger}
V_A^{\dagger} V_A \v$ we denote the respective eigenvectors with
the superscript symbols $1$ and $2$. By
Lemma~\ref{lem:idendityeigenvector} we have $V_{G}^1 = I_d =
V_A^{-1}V_A$ and Lemma~\ref{lem:eigenvectorrelation} declares that
$V_G^2 = V_A^{-1}V_{E^f}$ where $V_{E^f}$ is computed using the
canonical inner product in $\RA$. For the given $\alpha$ let $q =
\argmin_j \lambda_j(G) \leq \alpha$. The invariance of the embedded
subspace $\F_q$ to the underlying
scalar product signifies that the real span of the vectors $\{V_A
\v_d^1(G),\cdots, V_A \v_q^1(G) \}$ equal the real span of the vectors
$\{V_A \v_d^2(G),\cdots, V_A \v_q^2(G) \}$. As $\forall j
\in \{1,2,\cdots,d\}, V_A \v_j^1(G) = \v_j(A)$ and $V_A \v_j^2(G) =
\v_j(E^f)$, the result follows.
\end{proof}
\bibliographystyle{siam} \bibliography{RankDeficiency_KalmanFilter_MajorRevision-KG}
\end{document}